\documentclass[12pt,draftclsnofoot,onecolumn]{IEEEtran}

\ifCLASSINFOpdf
  \else  
\fi
\usepackage{amsmath}
\usepackage{mathrsfs}
\usepackage{amsthm}
\usepackage{multirow}
\usepackage{cite}    
\usepackage{color}
\usepackage{graphicx}      

\newtheorem{remark}{Remark}
\newtheorem{lemma}{Lemma}

\newtheorem{theorem}{Theorem}

\newcommand{\be}{\begin{eqnarray}}
\newcommand{\ee}{\end{eqnarray}}
\newcommand{\bi}{\begin{itemize}}
\newcommand{\ei}{\end{itemize}}
\newcommand{\ba}{\begin{array}}
\newcommand{\ea}{\end{array}}
\newcommand{\bm}{\begin{matrix}}
\newcommand{\eem}{\end{matrix}}

\newcommand{\no}{\nonumber}

\begin{document}
\title{Optimization Methods Rooting in Optimal Control }

\author{Huanshui~Zhang,~Hongxia~Wang
\thanks{This work was supported by the Original Exploratory Program Project of National Natural Science Foundation of China (62250056), the Joint Funds of the National Natural Science Foundation of China (U23A20325), the Major Basic Research of
Natural Science Foundation of Shandong Province (ZR2021ZD14), and the High-level Talent Team Project of Qingdao West
Coast New Area (RCTD-JC-2019-05).
}
}
\maketitle
\begin{abstract}
In the paper, we propose solving optimization problems (OPs) and understanding the Newton method from the optimal control view. We propose a new optimization algorithm based on the optimal control problem (OCP). The algorithm features converging more rapidly than gradient descent, meanwhile, it is superior to Newton's method because it is not divergent in general and can be applied in the case of a singular Hessian matrix. These merits are supported by the convergence analysis for the algorithm in the paper. We also point out that the convergence rate of the proposed algorithm is inversely proportional to the magnitude of the control weight matrix and proportional to the control terminal time inherited from OCP.  

\end{abstract}
%

%
\section{Introduction}
 
The research history of optimization problems is long and rich, spanning several centuries and evolving with the development of mathematics, engineering, operations research, and computer science. This originates from that optimization problems are widespread in the modeling of real-world systems for a very broad range of applications. Such applications include economies of scale, fixed charges, finance, allocation and location problems, structural optimization, engineering design, network and transportation problems, chip design and database problems, mechanical design, chemical engineering design and control, molecular biology, and several other combinatorial optimization problems such as integer programming and related graph problems. With the development and implementation of practical optimization algorithms, more and more scientists in diverse disciplines have been using optimization techniques to solve problems \cite{rao2019engineering, belegundu2019optimization,banichuk2013introduction,chong2023introduction}. 

Predominant optimization methods, such as the gradient descent \cite{faires2003numerical} and Newton's methods \cite{boyd2004convex}, are proposed based on the function approximation. Despite being powerful, favorable,  and widely used optimization algorithms, they also have limitations and potential drawbacks.
The learning rate is a crucial hyperparameter in gradient descent \cite{baydin2017online}. Gradient descent's performance is susceptible to the learning rate (If the learning rate is too small, the algorithm may converge slowly, while if it is too large, the algorithm may overshoot the minimum and fail to converge \cite{fei2020new}) while the hyperparameters tuning is challenging \cite{ruder2016overview}. Newton's method may diverge if the Hessian matrix is singular or the algorithm encounters numerical instability (when the Hessian matrix is ill-conditioned). Both gradient descent and Newton's methods are sensitive to the choice of the initial value. Suppose the initial guess is far from the actual root or minimum. In this case, Newton's method might converge to a different solution or fail to converge altogether and gradient descent may get stuck in local minima or saddle points. To mitigate some of the aforementioned issues, a body of variations and enhancements, such as stochastic gradient descent \cite{bottou2010large}, momentum \cite{qian1999momentum}, adaptive learning rate methods \cite{duchi2011adaptive,o2015adaptive,zeiler2012adadelta}, quasi-Newton method \cite{broyden1967quasi}, inexact Newton method \cite{dembo1982inexact}, and modified Newton method \cite{fletcher1977modified} have been developed.  However, it still seems somewhat blind to tune parameters in diverse methods.

The paper is the first to propose solving OPs and understanding the existing optimization methods from the optimal control view. Every parameter or variable in OCP has a clear physical meaning whereby it is easy to tune the parameters to alter the interested variables \cite{lewis2012optimal}. Accordingly, the convergence rate of algorithms based on OCP can be altered by adjusting those parameters inherited from OCP. It should be emphasized that in this paper, we do not apply the results of optimal control mechanically, but we propose a new optimization algorithm based on the optimal control and make a convergence analysis. The algorithm features converging more rapidly than gradient descent, meanwhile, it is superior to Newton's method because it can be applied in the case of a singular Hessian matrix. We also point out that the convergence rate of the proposed algorithm can be accelerated by decreasing the magnitude of the control weight matrix or increasing the control terminal time inherited from OCP.  

The remainder is arranged as follows. In Section II, we propose an update relation, an optimization method, for OP by the solution to OCP.  We explore the relationship between the existing (gradient descent and Newton's method) and proposed methods in Section III.  Section IV is devoted to the convergence analysis of explicitly implementing the proposed method. Some conclusions are achieved in Section V.


\section{Resolve optimization problem by using optimal control method}
Assume that $f(x)$ is twice continuously
differentiable. In most practical settings, one would almost surely need to resort to numerical
techniques to address the optimization problem
\begin{align} \min_x f(x). \label{obj}\end{align}
The traditional numerical method, for instance, gradient descent or Newton's method,  directly finds zeros of $f'(x)$ by an update recursion
\begin{align}
x_{k+1}=\Phi(x_k)
\end{align}
with $\Phi(\cdot)$ being a designed function. 

Instead, we will propose an updated relation with the aid of the following OCP:
\begin{align}
&\min_{u} \sum_{k=0}^N[f(x_k)+u_k'Ru_k]+f(x_{N+1}),\label{perf}\\
&\mbox{subject to~} x_{k+1}=x_k+u_k, \label{sys1}
\end{align}
where $x_k$ and $u_k$ are the state and control of system \eqref{sys1}, respectively, integer $N>0$ is the control terminal time, positive definite matrix $R$ is the control weight matrix,  and \eqref{perf} is the cost functional of OCP, closely related to OP \eqref{obj}.   

\begin{remark}
It can be seen from \eqref{perf}  that the optimization algorithm based on OCP will be fast because  $u_k$ in \eqref{sys1} is chosen to minimize $f(x_{k+1})$.   
\end{remark}

\begin{remark} According
to OCP \eqref{perf}-\eqref{sys1}, the smaller the weight matrix $R$ is, the more control energy is allowed to be
consumed for minimization, i.e., the larger the update step size $u_k$ in \eqref{sys1} is, the more rapidly the
iteration scheme \eqref{sys1} may converge to the extremal point of OP \eqref{obj}.
\end{remark}

\begin{remark}
OCP \eqref{perf}-\eqref{sys1} degenerates into $\min_{u_0} f(x_1)$ subject
to \eqref{sys1} without restriction of consumed control energy. 
\end{remark}

\begin{lemma}
The optimal control strategy of OCP \eqref{perf}-\eqref{sys1} admits
\begin{align}
u_k=-R^{-1}\sum_{i=k+1}^{N+1}f'(x_i). \label{conlaw}
\end{align}
\end{lemma}

\begin{proof}
The conclusion is direct by using the maximum principle. We thus omit its proof.
\end{proof}

Take \eqref{conlaw} as an update direction. Then it can be obtained from \eqref{sys1} an update relation
\begin{align}
x_{k+1}=x_k-R^{-1}\sum_{i=k+1}^{N+1}f'(x_i). \label{impite}
\end{align}
Although the relation \eqref{impite} is implicit, some explicit approximation implementations are possible.  We will give an explicit approximation recursion later.

It should be stressed that the proposal of a new update scheme is not our ultimate goal. We are also interested in revealing the relationship between existing methods and our methods.

\section{The proposed algorithm and convergence analysis}
To explore the relationship between the existing method and our method, we reformulate the update relation \eqref{impite} now.
Make the first-order Taylor expansion of $f'(x_i)$ at point $x_{i-1}$ in \eqref{impite}. It can be derived that 
\begin{align}
&x_{k+1}=x_{k}-g_{k}, k=0,\cdots, N,\label{rimpite}\\
&g_{k}=\frac{1}{R+f^{(2)}_{k}}(f'_{k}+\sum_{i=k+1}^{N}(f'_{i}-f^{(2)}_{i}g_{i})), g_N=\frac{f'_N}{R+f^{(2)}_N},\label{gk}
\end{align}
where $f'_k \dot =f'(x_k)$ and $f^{(2)}_{k}\dot=f^{(2)}(x_k)$.

To facilitate calculation, we replace all $x_i, i \ge k+1$, with $x_k$ in the right-hand side of \eqref{gk} and  obtain the following recursion
\begin{align}
&x_{k+1}=x_{k}-\bar g_{k}(x_k), k=0,\dots,N,\label{rrimpite}\\
&\bar g_{l}=\frac{1}{R+f^{(2)}_{k}}(f'(x)+R\bar g_{l+1}(x)), \bar g_N=\frac{f'(x)}{R+f^{(2)}(x)}, l=N-1,N-2,\dots, k.\label{bargk}
\end{align}


According to \eqref{rrimpite}-\eqref{bargk}, the following algorithm is proposed.

{\bf Algorithm 1}: \\
1). Fix integer $N>0$ and prescribe any $x_0$;\\
2). Carry out the update relation  \eqref{rrimpite}-\eqref{bargk} and obtain $x_{N+1}$;

\begin{remark}
In the case of $R=0$, the recursion \eqref{rrimpite}-\eqref{bargk} recovers the famous Newton method.
\end{remark}

In the sequel, we will make some convergence analysis pertinent to Algorithm 1. To compare with Newton's method,  we assume that $f(x)$ is cubic differentiable.

\begin{theorem}\label{thm1}
Assume $f'(x_*)=0$ and $f^{(2)}(x_*)> 0$. Consider the iteration scheme 
\begin{align}
x_{k+1}=&x_{k}-\bar g_{0}(x_k), \label{expiten}\\
\bar g_{l}(x)=&\frac{1}{R+f^{(2)}(x)}[f'(x)+R\bar g_{l+1}(x)], \bar g_N=\frac{f'(x)}{R+f^{(2)}(x)},l=N-1,\dots,0. \label{rrbargk}
\end{align}
 Then \eqref{expiten}-\eqref{rrbargk} is linearly convergent to $x_*$ and has convergence constant $(\frac{R}{R+f^{(2)}(x_*)})^{N+1}$.
\end{theorem}

\begin{proof}
The iteration error evolves as follows
\begin{align}
&x_{k+1}-x_*\no\\
=&x_{k}-\bar g_{0}(x_k)-x_*\no\\
=&\phi(x_k)-\phi(x_*)\no\\
=&\phi'(x_*)(x_k-x_*)+\phi^{(2)}(x_*)(x_{k}-x_*)^2+o(|x_{k}-x_*|^2).\label{generr}
\end{align}
In the above,
 \begin{align}\phi(x)=&x-\bar g_0(x),\\
\phi'(x_*)=&(\frac{R}{R+f^{(2)}(x_*)})^{N+1}. \label{phiderstar}
\end{align}
$\phi'(x_*)$ in \eqref{phiderstar} can be derived by induction over $\bar g_k'(x_*)$. In detail, 
according to \eqref{bargk},
\begin{align}
\bar g_{k}'(x)=&\frac{1}{R+f^{(2)}(x)}(f^{(2)}(x)+R\bar g'_{k+1}(x))-\frac{f^{(2)}(x)}{(R+f^{(2)}(x))^2}(f'(x)+R\bar g_{k+1}(x)), \label{bargkx}\\
 \bar g'_N(x)=&\frac{f^{(2)}(x)}{R+f^{(2)}(x)}-\frac{f'(x)f^{(3)}(x)}{(R+f^{(2)}(x))^2}. \label{bargnx}
\end{align}
\eqref{bargk} implies that every $\bar g_i(x)$ contains $f'(x)$ as a factor. Accordingly, if $f'(x_*)=0$, then
\begin{align}
\bar g_i(x_*)=&0,i=N,N-1,\cdots, 0 \label{gixstar}
\end{align}
which together with \eqref{bargkx}-\eqref{bargnx} yields
\begin{align}
\bar g_{k}'(x_*)=&\frac{1}{R+f^{(2)}(x_*)}(f^{(2)}(x_*)+R\bar g'_{k+1}(x_*)), \label{bargkxstar}\\
 \bar g'_N(x_*)=&\frac{f^{(2)}(x_*)}{R+f^{(2)}(x_*)}. \label{bargnxstar}
 \end{align}

Based on \eqref{bargkxstar} and \eqref{bargnxstar},  $\bar g'_i(x_*)=1-(\frac{R}{R+f^{(2)}(x_*)})^{N+1-i}$ can be acquired by using backward induction technique.

Firstly, in the case of $i=N$, it is direct that
\begin{align}
g_N'(x_*)=&\frac{f^{(2)}(x_*)}{R+f^{(2)}(x_*)}=1-\frac{R}{R+f^{(2)}(x_*)}.\end{align}

Secondly, assume for all $i\ge k+1$, there always hold
\begin{align}
g'_{i}(x_*)=1-(\frac{R}{R+f^{(2)}(x_*)})^{N+1-i}. \label{gidexstar}
\end{align}

Thirdly, we will show that in the case of $i=k$, \eqref{gidexstar} also holds. 
\begin{align}
\bar g_{k}'(x_*)=&\frac{1}{R+f^{(2)}(x_*)}(f^{(2)}(x_*)+R\bar g'_{k+1}(x_*))\no\\
=&\frac{1}{R+f^{(2)}(x_*)}[f^{(2)}(x_*)+R(1-(\frac{R}{R+f^{(2)}(x_*)})^{N-k})]\no\\
=&1-(\frac{R}{R+f^{(2)}(x_*)})^{N-k+1}. \label{bargkd}
  \end{align}
Now from \eqref{expiten}, 
\begin{align}
\phi'(x_*)=1-\bar g'_{0}(x_*)=(\frac{R}{R+f^{(2)}(x_*)})^{N+1}.
\end{align}

\end{proof}

\begin{remark}\label{rem4}
According to the convergence constant $(\frac{R}{R+f^{(2)}})^{N+1}$ provided in Theorem 1, the convergence rate of  \eqref{expiten}-\eqref{rrbargk}  depends on two adjustable parameters, $R$ and $N$. For given $N$, the smaller $R$, the more rapidly \eqref{expiten}-\eqref{rrbargk} converges. For given $R$, the larger $N$, the more rapidly \eqref{expiten}-\eqref{rrbargk} converges.
\end{remark}

\begin{remark}\label{remalmostquadraticcon}
Theorem 1 states that the explicit approximation \eqref{expiten}-\eqref{rrbargk} has linear convergence. However, the convergence constant $(\frac{R}{R+f^{(2)}})^{N+1}$ also implies that \eqref{expiten}-\eqref{rrbargk} can converge rapidly enough and almost achieve a quadratic convergence by adjusting $R$ and $N$ to decrease the convergence constant. 
\end{remark}

By direct manipulation, $\phi^{(2)}(x_*)$ in \eqref{generr} can be given as 
\begin{align}
&\phi^{(2)}(x_*)=\bar g^{(2)}_0(x_*)\no\\
=&2(N+1)(\frac{R}{R+f^{(2)}(x_*)})^{N+1}-\frac{R+f^{(2)}(x_*)}{f^{(2)}(x_*)}[1-(\frac{R}{R+f^{(2)}(x_*)})^{N+1}]\no\\
=&\frac{f^{(3)}(x_*)}{f^{(2)}(R+f^{(2)}(x_*))}\{[(2N+3)f^{(2)}(x_*)+R](\frac{R}{R+f^{(2)}(x_*)})^{N+1}-R-f^{(2)}(x_*)\}\no\\
=&\frac{f^{(3)}(x_*)}{f^{(2)}(x_*)}\frac{[(2N+3)f^{(2)}(x_*)+R]R^{N+1}-(R+f^{(2)}(x_*))^{N+2}}{(R+f^{(2)}(x_*))^{N+2}}.
\end{align}
On account of $|[(2N+3)f^{(2)}(x_*)+R]R^{N+1}-(R+f^{(2)}(x_*))^{N+2}|<|(R+f^{(2)}(x_*))^{N+2}|$, we have 
$|\phi^{(2)}(x_*)|<|\frac{f^{(3)}(x_*)}{f^{(2)}(x_*)}|$. 
\begin{remark}\label{rem6}
$|\phi^{(2)}(x_*)|<|\frac{f^{(3)}(x_*)}{f^{(2)}(x_*)}|$ together with \eqref{generr} and Remark \ref{remalmostquadraticcon} implies that 
 the explicit approximation \eqref{expiten}-\eqref{rrbargk} even converges more rapidly than Newton's method as $N\rightarrow +\infty$.
\end{remark}

Remark \ref{rem4}-\ref{rem6} are devoted to the convergence of \eqref{expiten}-\eqref{rrbargk}. It can be regarded that \eqref{expiten}-\eqref{rrbargk} enjoys a unified step size $\bar g_0(k)$ in each iteration because of the unified convergence rate. As in the gradient descent method \cite{zhao2019research,li2018adaerror}, the unified step size may lead to a slow convergence. 
%
%
For this, Algorithm 1 with varying step size $\bar g_k(x_k)$ is a better alternative to the recursion \eqref{expiten}-\eqref{rrbargk}.

To analyze the convergence rate of Algorithm 1, for $i=1,\cdots, N$, let
\begin{align}
\psi_i(x)=x-\bar g_i(x). \label{psi}
\end{align}
Then the following theorem holds.

\begin{theorem}\label{thm2}
Assume $f^{(2)}(x_*)>0$. The iteration error of Algorithm 1 obeys
 \begin{align}
|x_{k+1}-x_*|\le | \psi'_{k}(x_*)(x_{k}-x_*)|\le |\prod_{i=0}^{k}\psi'_i(x_*)(x_0-x_*)|, \label{rate}
\end{align}
  where
\begin{align}
\psi_i'(x_*)=[\frac{R}{R+f^{(2)}(x_*)}]^{N-i+1}.\label{psid}
\end{align} 
\end{theorem}
\begin{proof}
The proof is a straightforward derivation from \eqref{psi}. 
\end{proof}

\begin{remark}
Compared to \eqref{expiten}-\eqref{rrbargk}, Algorithm 1 has varying step sizes $\bar g_k(x_k)$ at each iteration. From Theoren \ref{thm1} and \eqref{bargkd}, the algorithm holds a greater step size in the initial stages and a smaller one in the latter stages, which guarantees rapid convergence and lower computational cost simultaneously.
\end{remark}

\begin{remark}
On one side, it follows from \eqref{rate} and \eqref{psid} that Algorithm 1 superlinearly converges to the minimal of $f(x)$ as $f^{(2)}(x_*)>0$. On the other side, Algorithm 1 still works even if there occasionally holds $f^{(2)}(x)=0$ because of the positive definite matrix $R$.
\end{remark}

%

%

\section{Conclusion}
In the paper, a novel algorithm has been developed based on OCP. The algorithm converges more rapidly than gradient descent, meanwhile, unlike Newton's method, it still works even if the singular Hessian matrix appears in iterate. We also pointed out that Newton's method is recovered when the control weight matrix disappears in the developed algorithm. The convergence rate of the algorithm depends on two adjustable parameters. Because of their clear physical meanings inherited from OCP,  it is very convenient to tune them to speed up the convergence rate.

\end{document}